\documentclass[a4paper,11pt]{amsart}
\title{On the topology of random real complete intersections}
\author{Michele Ancona}
\thanks{Institut de Recherche Math\'ematique Avanc\'ee, Universit\'e de Strasbourg.\\
 \emph{E-mail address}: \url{michele.ancona@math.unistra.fr}. }
\date{}
\usepackage[a4paper,bindingoffset=0.2in,%
            left=1in,right=1in,top=1in,bottom=1in,%
            footskip=.25in]{geometry}
\usepackage[latin1]{inputenc}
\usepackage[T1]{fontenc}
\usepackage[english]{babel}
\usepackage{amsmath,amssymb,amsthm}
\usepackage{amsfonts}%
\usepackage{amssymb}%
\usepackage{indentfirst}
\usepackage{fancyhdr}
\usepackage{physics}
\usepackage{young}
\usepackage{graphicx}

\usepackage[vcentermath]{youngtab}

\RequirePackage[colorlinks,linkcolor=blue,citecolor=blue,urlcolor=blue]{hyperref}
\usepackage{enumerate}
\usepackage{color}
\usepackage{pst-fill,pst-grad,pst-plot,pst-eucl,pstricks-add,pst-node}
\usepackage{mathdots}
\usepackage{dsfont}

\theoremstyle{plain}
\newtheorem{thm}{Theorem}[section]
\newtheorem{lemma}[thm]{Lemma}
\newtheorem{prop}[thm]{Proposition}
\newtheorem{cor}[thm]{Corollary}
\newtheorem{oss}[thm]{Remark}
\newtheorem{example}[thm]{Example}

\theoremstyle{definition}
\newtheorem{defn}[thm]{Definition}
\newtheorem{conv}[thm]{Notation}

\renewcommand{\P}{\mathbf{P}}

\newcommand{\R}{\mathbf{R}}
\newcommand{\C}{\mathbf{C}}

\newcommand{\crit}{\textrm{crit}}

\newcommand{\Vol}{\textrm{Vol}}

\allowdisplaybreaks

\begin{document}
\maketitle
\begin{abstract} 
Given a real projective variety $X$ and $m$ ample  line bundles $L_1,\dots L_m$ on $X$ also defined over $\R$, we study the topology of the real locus of the complete intersections defined by global sections  of $L_1^{\otimes d}\oplus\cdots\oplus L^{\otimes d}_m$. We prove that the Gaussian measure of  the space of sections  defining real complete intersections with high total Betti number (for example, maximal complete intersections) is exponentially small, as $d$ grows to infinity. This is deduced by proving that, with very high probability, the real locus of a  complete intersection defined by a section  of $L_1^{\otimes d}\oplus\dots\oplus L^{\otimes d}_m$ is isotopic to the real locus of a complete intersection of smaller degree. 
\end{abstract}
\section{Introduction}
The goal of this paper is to study the topology of  real complete intersections inside a real algebraic variety. 
More precisely, we are interested in the study of the Betti numbers of their real loci, as the degree of the complete intersections goes to infinity. 
As the complex locus of such complete intersections gives restrictions on the topology of their real locus, we will start by recalling what happens in the complex case before passing to the real one. 
\subsection{Topology of complex complete intersections}
Let $X$ be a smooth complex projective variety of dimension $n$ and $L_1,\dots,L_m$ be ample line bundles over $X$. For any $d>0$, we will denote by $L_i^d$  the $d$-th tensor power of $L_i$. 
 Let us denote by $\Delta_d$ the discriminant locus in $H^0(X,\oplus_{i=1}^mL_i^d)$, that is, the space of sections of $\oplus_{i=1}^mL_i^d$ that do not vanish transversally (or, equivalently, the space of sections $s$ whose zero locus $Z_s$ is singular).  By Bertini's theorem, for $d$ large enough, the discriminant locus $\Delta_d$ is a complex algebraic hypersurface of $H^0(X,\oplus_{i=1}^mL_i^d)$ and then its complement $H^0(X,\oplus_{i=1}^mL_i^d)\setminus \Delta_d$ is connected. Therefore, given any pair of sections $s,s'\in H^0(X,\oplus_{i=1}^mL_i^d)\setminus \Delta_d$, one can always find a path in $H^0(X,\oplus_{i=1}^mL_i^d)\setminus \Delta_d$ joining them and, by Ehresmann's theorem, an isotopy between  their zero loci $Z_s$ and $Z_{s'}$. In particular, this implies that the zero loci $Z_s$ and $Z_{s'}$ are  diffeomorphic and, then, that their Betti numbers are the same. As a consequence, the total Betti number $b_*(Z_s)$ of the zero locus of a  section $s\in H^0(X,\oplus_{i=1}^mL_i^d)\setminus \Delta_d$ only depends on $d$ and on $L_1,\dots,L_m$, and not on the  choice of $s$. In particular, one can compute the value $b_*(Z_s)$, for $s\in H^0(X,\oplus_{i=1}^mL_i^d)\setminus \Delta_d$, as a function of $d$ and find  the asymptotics 
\begin{equation}\label{ineq in intro}
b_*(Z_s)=\mathrm{v}d^n+O(d^{n-1}),
\end{equation}
 where $\mathrm{v}=\mathrm{v}(L_1,\dots,L_m)$ is a topological constant only depending on the line bundles $L_i$, see Lemma \ref{totalBetti}.
\subsection{Topology of real complete intersections} Let us now suppose that $X$ is defined over $\R$. By this, we mean that the complex variety $X$ is equipped with an anti--holomorphic involution $c_X:X\rightarrow X$, called the \emph{real structure}. We call the pair $(X,c_X)$ a \emph{real algebraic variety.} For example, the complex projective space $\P^n$ equipped with the standard conjugaison $\textit{conj}:x\in\P^n\mapsto\bar{x}\in\P^n$ is a real algebraic variety. More generally, the solutions of a system of homogeneous real polynomial equations in $n+1$ variables define a real algebraic variety $X$ inside $\P^n$, whose real structure  is the restriction of $\textit{conj}$ to $X$.
The \emph{real locus} $\R X$ of a real algebraic variety is the set  of fixed points of the real structure, that is $\R X=\mathrm{Fix}(c_X)$. It is either empty or a finite union of $n$-dimensional smooth manifolds. 
 
 Let us also suppose that $L_1,\dots,L_m$ are real holomorphic line bundles, that is, they are equipped with real structures $c_{L_1},\dots,c_{L_m}$ that are linear anti-holomorphic in the fibers and such that $\pi_i\circ c_{L_i}=c_X\circ \pi_i$, where $\pi_i:L_i\rightarrow X$ is the natural projection. 
 
 For any $d>0$, let $\R H^0(X,\oplus_{i=1}^mL_i^{d})$ be the space of real sections of $\oplus_{i=1}^mL_i^{ d}$, that is, the space of sections $s=(s_1,\dots,s_m)\in H^0(X,\oplus_{i=1}^mL_i^{ d})$ such that $s_i\circ c_X=c_{L_i^d}\circ s_i$, for any $i\in\{1,\dots,m\}$. 
 Let $\R \Delta_d$ be the real discriminant locus, that is the space of real sections that do not vanish transversally.  By Bertini theorem, $\R\Delta_d$ is a real hypersurface in $\R H^0(X,\oplus_{i=1}^mL_i^{ d})$. In contrast to the complex setting, $\R H^0(X,\oplus_{i=1}^mL_i^{d})\setminus \R\Delta_d$ is \emph{not} connected and this produces the following phenomenon: the topology of the real locus $\R Z_s$ of a real section $s\in\R H^0(X,\oplus_{i=1}^mL_i^{d})$ depends on the choice of the section. Indeed, the real discriminant creates walls inside $\R H^0(X,\oplus_{i=1}^mL_i^{ d})$, and the Betti numbers of the real locus of a complete intersection change when we cross a wall.

 This raises a natural question: what is the topology of $\R Z_s$, if we pick $s$ at random?
 This question is moreover motivated by the fact that the number of connected components of $\R H^0(X,\oplus_{i=1}^mL_i^{d})\setminus \R\Delta_d$ grows very fast as $d\rightarrow\infty$ and then a deterministic study of all the topologies seems unproachable. For example, the number of connected component of $\R H^0(\P^n,\mathcal{O}(d))\setminus \R\Delta_d$ grows super-exponentially in $d$, see \cite{orevkov}.
 \medskip

The main result of this paper (Theorem \ref{theorem approximation}) is that, with very high probability, the real locus of  $Z_s$, with $s$ a real section of $\oplus_{i=1}^mL_i^d$, is diffeomorphic to the real locus of $Z_{s'}$, with $s'$ a (well--chosen) real section of $\oplus_{i=1}^m L_i^{d'}$, for a  sufficiently smaller $d'<d$. As a consequence we will prove that complete intersections whose real loci have high total Betti number are very rare (Theorem \ref{theorem rarefaction}).
\subsection{Probability measure} In order to state our main results, let us introduce the probability measure we consider. We  equip each real holomorphic line bundle $L_i$ with a smooth  Hermitian metric $h_i$ that is real (meaning $c_{L_i}^*h_i=\bar{h}_i$) and of positive curvature $\omega_i$.  The space of real global sections $\R H^0(X,\oplus_{i=1}^mL_i^d)$ is then naturally equipped with a   $\mathcal{L}^2$-scalar product defined by
\begin{equation}\label{l2 scalar}
\langle s,s'\rangle_{\mathcal{L}^2}=\displaystyle\sum_{i=1}^m\int_X h_i^d(s_i,s_i')\frac{\omega_i^{\wedge n}}{n!}
\end{equation}
for any pair of real global sections $s=(s_1,\dots,s_m)$ and $s'=(s_1',\dots,s_m')$ in $\R H^0(X,\oplus_{i=1}^mL_i^d)$, where $h_i^d$ is the real Hermitian metric on $L_i^d$ induced by $h_i$.
In turn, the $\mathcal{L}^2$-scalar product \eqref{l2 scalar} naturally induces   a Gaussian probability measure $\mu_d$ defined by
\begin{equation}\label{gaussian measure}
\mu_{d}(A)=\frac{1}{\sqrt{\pi}^{N_d}}\int_{s\in A}e^{-\norm{s}^2_{\mathcal{L}^2}}\mathrm{d}s
\end{equation}
for any open set $A\subset \R H^0(X,\oplus_{i=1}^m L_i^d)$, where $N_d$ is the dimension of $\R H^0(X,\oplus_{i=1}^mL_i^d)$ and $\mathrm{d}s$ the Lebesgue measure induced by the $\mathcal{L}^2$-scalar product \eqref{l2 scalar}. 

 The probability space we will consider is then $\big(\R H^0(X,\oplus_{i=1}^mL_i^d),\mu_d\big)$. A random section $s=(s_1,\dots,s_m)\in\R H^0(X,\oplus_{i=1}^mL_i^d)$ gives us a random real subvariety $Z_s=Z_{s_1}\cap\dots\cap Z_{s_m}$ with real locus $\R Z_s=\R Z_{s_1}\cap\dots\cap \R Z_{s_m}=Z_s\cap \R X$.
\begin{example}[Kostlan polynomials]\label{Kost} When $(X,c_X)$ is the $n$-dimensional projective space and $(L,c_L,h)$ is the degree $1$ real holomorphic line bundle  equipped with the standard Fubini-Study metric, then the vector space $\R H^0(X,L^{d})$ is isomorphic to the space $\R_d^{hom}[X_0,\dots,X_n]$ of degree $d$ homogeneous real polynomials in $n+1$ variables and the $\mathcal{L}^2$-scalar product is the one which makes the family of monomials $\big\{\sqrt{\binom{(n+d)!}{n!\alpha_0!\cdots\alpha_n!}}X_0^{\alpha_0}\cdots X_n^{\alpha_n}\big\}_{\alpha_0+\dots+\alpha_n=d}$ an orthonormal basis. 
Up to scalar multiplication, this is the only scalar product on $\R_d^{hom}[X_0,\dots,X_n]$ which is invariant by the action of the orthogonal group $O(n+1)$ (acting on the variables $X_0,\dots,X_n$) and such that the standard monomials are orthogonal to each other.

A random polynomial with respect to the Gaussian probability measure induced by this scalar product is called a Kostlan polynomial \cite{ko,ss}.
\end{example}

\subsection{Statements of the main results} Let us state the main result of the paper.
\begin{thm}\label{theorem approximation} Let $X$ be a real algebraic variety  and $L_1,\dots,L_m$ be   real Hermitian line bundles  of positive curvature. 
\begin{enumerate}
\item There exists a positive $\alpha_0<1$ such that for any $\alpha_0< \alpha <1$ the following happens:  the probability that,  for a real  section $s\in \R H^0(X,\oplus_{i=1}^mL_i^d)$, there exists a real section $s'\in \R H^0(X,\oplus_{i=1}^mL_i^{\lfloor  \alpha d \rfloor})$ such that the pairs $(\R X,\R Z_s)$ and $(\R X,\R Z_{s'})$ are isotopic,  is at least $1-O(d^{-\infty})$, as $d\rightarrow\infty$. (Here, the notation $O(d^{-\infty})$ stands for $O(d^{-k})$ for any $k\in\mathbb{N}$.)
\item  For any $k\in \mathbb{N}$ there exists a positive constant $c$ such that the following happens: the probability that, for a real section $s\in \R H^0(X,\oplus_{i=1}^mL_i^d)$, there exists a real section $s'\in \R H^0(X,\oplus_{i=1}^mL_i^{d-k})$ such that the pairs $(\R X,\R Z_s)$ and $(\R X,\R Z_{s'})$ are isotopic,  is at least $1-O(e^{-c\sqrt{d}\log d})$, as $d\rightarrow\infty$. If moreover the real Hermitian metrics on $L_1,\dots,L_m$ are  analytic, this probability is  at least $1-O(e^{-cd})$, as $d\rightarrow\infty$
\end{enumerate}
\end{thm}
Hence, in the sense of measure, most topologies of the real locus of the intersection  of degree $d$ real ample divisors  can be found in lower degree. 
Let us stress that this is a real phenomenon: for $d$ large enough,  the complex loci $Z_s$ and $Z_s'$ of sections of $L^{d}$ and $L^{\lfloor \alpha d \rfloor}$, $\alpha<1$, are not diffeomorphic (because, by \eqref{ineq in intro}, we have $b_*(Z_s)>b_*(Z_{s'})$), while, by Theorem \ref{theorem approximation}, their real loci are diffeomorphic, and even isotopic, with very high probability.

For Kostlan polynomials (see Example \ref{Kost}), such approximation was proved in \cite{breiding,diatta},
 while for $m=1$, that is for random real hypersurfaces in a general real algebraic variety $X$, Theorem \ref{theorem approximation} coincides with \cite[Theorem 1.4]{anc5}.  Theorem \ref{theorem approximation} is then  natural generalization of these results for real complete intersections in a real algebraic variety $X$.
 
\begin{oss} In Theorem \ref{theorem approximation}, one can also allow different tensor powers of the line bundles, that is, one can consider $L_1^{d_1}\oplus\dots\oplus L_m^{d_m}$. In this case, one should set $d:=\max\{d_1,\dots,d_m\}$ and Theorem \ref{theorem approximation}(1)  becomes as follows: for any $\alpha_0<\alpha<1$, the real locus of a real section of $L_1^{d_1}\oplus\dots\oplus L_m^{d_m}$ is isotopic to the real locus of a real section of $L_1^{\min\{\alpha d,d_1\}}\oplus\dots\oplus L_m^{\min\{\alpha d,d_m\}}$,  with  probability at least $1-O(d^{-\infty})$, as $d$ goes to infinity. (Theorem \ref{theorem approximation}(2) also has an analogous statement in this case.)

 The proof of this slightly more general version of Theorem \ref{theorem approximation} is the same as the one that we present in the article. We decided to work with $d_1=\cdots=d_m=d$  for  clarity of exposition and in order to avoid a heavy notation.

\end{oss}

Let us now explain one consequence of Theorem \ref{theorem approximation}. Recall that by the Smith-Thom inequality \cite{thomreelle},  the total Betti number of the real locus $\R X$ of a real algebraic variety is bounded from above by the total Betti number of its complex locus:  
\begin{equation}\label{smith-thom}
\sum_{i=0}^n\dim H_i(\R X, \mathbf{Z}/2)\leq \sum_{i=0}^{2n}\dim H_i(X, \mathbf{Z}/2).
\end{equation}
We will more compactly write $b_*(\R X)\leq b_*(X)$, where $b_*$ denotes the total Betti number with $\mathbf{Z}/2$-coefficients. 
For a algebraic curve $C$, Smith-Thom inequality is known as Harnack-Klein inequality \cite{harnack, klein} and reads 
$$b_0(\R C)\leq g(C)+1$$
where $g(C)$ denotes the genus of $C$.

Putting together Smith-Thom inequality \eqref{smith-thom} and the asymptotics \eqref{ineq in intro} for the total Betti number of a complete intersection, we find 
$$b_*(\R Z_s)\leq \mathrm{v}(L_1,\dots,L_m)d^n+O(d^{n-1}),$$
for any $s\in\R H^0(X,\oplus_{i=1}^mL_i^d)$.

\begin{thm}\label{theorem rarefaction} Let $X$ be a real algebraic variety of dimension $n$ and $L_1,\dots,L_m$ be a  real  Hermitian line bundles  of positive curvature.
\begin{enumerate}
\item For any $\epsilon>0$ small enough, we have 
$$\mu_d\big\{s\in\R H^0(X,\oplus_{i=1}^mL_i^d), b_*(\R Z_s)\geq  (1-\epsilon)b_*(Z_s)\big\}\leq O(d^{-\infty})$$
as $d\rightarrow\infty$. (Here, the notation $O(d^{-\infty})$ stands for $O(d^{-k})$ for any $k\in\mathbb{N}$.)
\item For any $a>0$  there exists $c>0$ such that 
$$\mu_d\big\{s\in\R H^0(X,\oplus_{i=1}^mL_i^d), b_*(\R Z_s)\geq b_*( Z_s)-ad^{n-1}\big\}\leq O(e^{-c\sqrt{d}\log d})$$
as $d\rightarrow\infty$.
Moreover, if the real Hermitian metrics on $L_1,\dots,L_m$ are  analytic,  then
$$\mu_d\big\{s\in\R H^0(X,\oplus_{i=1}^mL_i^d), b_*(\R Z_s)\geq b_*( Z_s)-ad^{n-1}\big\}\leq O(e^{-cd}).$$
\end{enumerate}
Here, the measure $\mu_d$ is the Gaussian measure defined in Equation \eqref{gaussian measure}. 
\end{thm}



Hence, real algebraic complete intersections in $X$ with high total Betti number are very rare. For the case of real curves in real algebraic surfaces, such result  was proved in \cite{gwexp} using different techniques from those in this article, in particular using the theory of laminary currents. In our case, Theorem \ref{theorem rarefaction}  is a consequence of the low degree approximation property given by Theorem \ref{theorem approximation}. This has already been observed first in the case of Kostlan complete intersections  in $\P^n$ in \cite{breiding,diatta} and then in the case of real hypersurfaces in a real algebraic variety in \cite{anc5}. 
\medskip  

Finally, let us recall that the expected Betti numbers of the real locus of a degree $d$ random complete intersection in a real $n$-dimensional algebraic variety $X$ are of order $d^{n/2}$, as $d\rightarrow\infty$, see \cite{gw2,gw3,gwlowerbound}.

\subsection{Idea of the proof of Theorem \ref{theorem approximation}} Let us give a sketch of proof of Theorem \ref{theorem approximation}\textit{(2)}. We first define a map $\R H^0(X,\oplus_{i=1}^mL^d_i)\rightarrow\R H^0(X,\oplus_{i=1}^mL^{d-k}_i)$, which will serve as "low degree approximation map". This map is constructed as follows. First, we fix once for all a real holomorphic section $\sigma=(\sigma_1,\dots,\sigma_m)$ of $\oplus_{i=1}^mL^{k}_i$, $k\in 2\mathbb{N}$, such that each $Z_{\sigma_i}$ is a smooth hypersurface with empty real locus. Then, we consider the $L^2$-orthogonal decomposition $\R H^0(X,\oplus_{i=1}^mL^d_i)=\R H_{\sigma}\oplus\R H_\sigma^{\perp}$, where $\R H_\sigma$ is the space of section of $\R H^0(X,\oplus_{i=1}^mL^d_i)$ which can be written in the form $(\sigma_1\otimes s'_1,\dots,\sigma_m\otimes s'_m)=:\sigma\otimes s'$. Then, every section $s$ can be  uniquely decomposed as $s=\sigma\otimes s'+s^{\perp}$, with $s^{\perp}\in \R H_\sigma^{\perp}$. Using this decomposition, the approximation map $\R H^0(X,\oplus_{i=1}^mL^d_i)\rightarrow\R H^0(X,\oplus_{i=1}^mL^{d-k}_i)$ we were looking for is $s\mapsto s'$. In order to prove that $\R Z_{s}$ is isotopic to $\R Z_{s'}$ with very high probability, we will use two arguments, one being deterministic and the other being probabilistic.
\begin{itemize}
\item The deterministic part consists in proving that the $\mathcal{C}^1$-norm of the section $s^{\perp}\in\R H_{\sigma}^{\perp}$ is exponentially small along $\R X$ (Propositions \ref{partialbergman} and \ref{logbergman}). This will be proved using the logarithmic Bergman kernel theory. Indeed we show that the $\mathcal{C}^1$-norm of the section $s^{\perp}$ concentrates around $\bigcup_{i=1}^mZ_{\sigma_i}$, that is, it is exponentially small outside a $1/\sqrt{d}$-neighborhood of $\bigcup_{i=1}^mZ_{\sigma_i}$. As each $Z_{\sigma_i}$ is disjoint from $\R X$, we deduce that the $\mathcal{C}^1(\R X)$-norm of $s^{\perp}$ is exponentially small. In this part of the proof we adapt to complete intersections the techniques developped in \cite{anc5} for hypersurfaces.
\item The probabilistic part consists in proving that, with very high probability, the $\mathcal{C}^1$-norm of $s$ is big enough along $\R X$.  This is proved by showing that, with very high probability, the distance from $s$ to the real discriminant is big enough (Proposition \ref{rapiddecay}) and then by relating this distance with the $\mathcal{C}^1(\R X)$-norm of $s$ (Proposition \ref{distance to the discriminant}). 
This  idea comes from  \cite{diatta} and was already used in \cite{anc5}. However, for the case of real complete intersections of a real algebraic variety, the techniques involved to prove these points are harder (even with respect to the one used in \cite{anc5} for the case of hypersurfaces of a real algebraic variety).  These arguments are developped in Sections \ref{Secdiscr} and \ref{secdistanceto} and could be of independent interest.
\end{itemize}
Putting together the previous two points, we get that $s$ is a $\mathcal{C}^1(\R X)$-small perturbation of $\sigma\otimes s'$  with very high probability . By Thom's isotopy lemma, this implies that $\R Z_s$ is isotopic to $\R Z_{\sigma\otimes s'}=\R Z_{s'}$, where the last equality follows from the fact that $\R Z_{\sigma}=\emptyset$.  
\subsection{Organization of the paper} The paper is organized as follows. In Section \ref{Secdiscr}, we study the topology of the complex complete intersection defined by a generic section of $H^0(X,\oplus_{i=1}^mL_i^d)$, $d\gg 1$. Using this, we compute the asymptotic of the degree of the discriminant locus $\Delta_d\subset H^0(X,\oplus_{i=1}^mL_i^d)$. In Section \ref{secdistanceto}, we study the function "distance to the real discriminant" defined on $\R H^0(X,\oplus_{i=1}^mL_i^d)$. This is done using some peak sections introduced in Section \ref{sec eval}.
In Section \ref{secorthogonaldecomposition}, we define and study a $\mathcal{L}^2$-orthogonal decomposition of $\R H^0(X,\oplus_{i=1}^mL_i^d)$ which plays a key role in the proof of the main results.
Finally, in Section \ref{SecProof}, we prove our main results, namely Theorems \ref{theorem approximation} and \ref{theorem rarefaction}.
\subsection*{Acknowledgements} Part of this article was written while I was in Tel Aviv and was supported  by the Israeli Science Foundation through the ISF Grants 382/15 and 501/18. I warmly thank Tel Aviv University for the excellent working conditions offered to me during the postdoc. 
\section{Complete intersections, Lefschetz pencils and degree of the discriminant}\label{Secdiscr}
Throughout this section, let $X$ be a complex projective variety  and $L_1,\dots,L_m$ be ample line bundles on $X$. In Section \ref{secAsymptTopol}, we  study the topology of the complex complete intersection $Z_{s_1}\cap\dots\cap Z_{s_m}$, where $s_i$ is a generic section of $L_i^d$ and $d$ is large. In Section \ref{seclefanddeg}, we study the discriminant hypersurface $\Delta_d\subset H^0(X,\oplus_{i=1}^mL_i^d) $ and, in particular, we compute its degree.
\subsection{Asymptotic topology of complete intersections}\label{secAsymptTopol} Let $\Delta_d\subset H^0(X,\oplus_{i=1}^mL_i^d)$ be the discriminant, that is the subset of sections $s=(s_1,\dots,s_m)$ which do not vanish transversally.  By Bertini's theorem, for $d$ large enough, the discriminant is an algebraic hypersurface and then, by Ehresmann's theorem,  the zero loci $Z_s$ and $Z_{s'}$ of two different sections $s,s'\in H^0(X,\oplus_{i=1}^mL_i^d)\setminus \Delta_d$ are diffeomorphic (and in fact, isotopic). In particular, the Euler characteristic and the total Betti number of a generic section of  $\oplus_{i=1}^mL_i^d$ only depend on $d$. The aim of this section is to give the asymptotics of such quantities, as $d\rightarrow\infty$.
\begin{prop}\label{euler} Let $L_1,\dots,L_m$ be ample line bundles over a complex projective  variety $X$ of dimension $n$. Let $s_1,\dots,s_m$ be generic holomorphic sections of $L^d_1,\dots,L^d_m$ and denote by $Z_{s_1},\dots,Z_{s_m}$ their vanishing loci. Then, as $d\rightarrow\infty$, we have the following asymptotic for the Euler characteristic of the complete intersection $Z_{s_1}\cap\dots\cap Z_{s_m}$:
$$\chi(Z_{s_1}\cap\cdots\cap Z_{s_m})=(-1)^{n-m}d^n\displaystyle\sum_{\substack{i_1+\dots+i_m=n-m \\ i_j\geq 0,\hspace{1mm} j\in\{1,\dots,m\}}}\int_Xc_1(L_1)^{i_1+1}\wedge\cdots\wedge c_1(L_m)^{i_m+1}+O(d^{n-1}).$$
\end{prop}
\begin{proof} 
First remark that if $L$ is a line bundle on $X$ and if $Y$ is an hypersurface defined by a section of $L$, then the adjunction formula gives us the following equality between Chern classes:
\begin{equation}\label{Adjunction1}
c_{j}(Y)=\sum_{i=0}^{j}(-1)^ic_1(L)_{\mid Y}^i\wedge c_{j-i}(X)_{\mid Y}.
\end{equation}
Let us denote by $Y_i:=Z_{s_1}\cap\cdots\cap Z_{s_i}$, so that we obtain the chain of inclusions $$Y_1\supset Y_2\supset\dots\supset Y_m.$$
We will use several times Equation \eqref{Adjunction1} for this chain of subvarieties. 
We start by applying  \eqref{Adjunction1}  for $j=n-m$, $Y=Y_m$, $X=Y_{m-1}$ and $L=L_{m_{\mid Y_{m-1}}}^d$  and obtain
\begin{equation}\label{Adjunction2}
c_{n-m}(Y_m)=\sum_{i_m=0}^{n-m}(-1)^{i_m}d^{i_m}c_1(L_m)_{\mid Y_m}^{i_m}\wedge c_{n-m-i_m}(Y_{m-1})_{\mid Y_m}.
\end{equation}
Remark that $\chi(Y_m)=\int_{Y_m}c_{n-m}(Y_m)$, so that from \eqref{Adjunction2} we obtain 
\begin{multline}\label{Adjunction3}
\chi(Y_m)=\int_{Y_m}\sum_{i_m=0}^{n-m}(-1)^{i_m}d^{i_m}c_1(L_m)_{\mid Y_m}^{i_m}\wedge c_{n-m-i_m}(Y_{m-1})_{\mid Y_m} \\
=\sum_{i_m=0}^{n-m}(-1)^{i_m}d^{i_m}\int_{Y_{m-1}}c_1(L_m)_{\mid Y_{m-1}}^{i_m+1}\wedge c_{n-m-i_m}(Y_{m-1})
\end{multline}
where in the second equality we used the identity $\int_{Y_m}\alpha_{\mid Y_m}=\int_{Y_{m-1}}\alpha\wedge c_1(L_m)_{\mid Y_{m-1}}$ for any closed form $\alpha$ on $Y_{m-1}$. 

Applying  \eqref{Adjunction1} to $j=n-m-i_m$, $Y=Y_{m-1}$, $X=Y_{m-2}$ and $L=L^d_{{m-1}_{\mid Y_{m-2}}}$, we have that  \eqref{Adjunction3} equals
\begin{multline*}
\sum_{i_m=0}^{n-m}\sum_{i_{m-1}=0}^{i_m}(-1)^{i_m+i_{m-1}}d^{i_m+i_{m-1}}\int_{Y_{m-1}}c_1(L_m)_{\mid Y_{m-1}}^{i_m+1}\wedge c_1(L_{m-1})_{\mid Y_{m-1}}^{i_{m-1}}\wedge c_{n-m-i_m-i_{m-1}}(Y_{m-2})_{\mid Y_{m-1}}
\end{multline*}
\begin{equation}\label{Adjunction4}
=\sum_{i_m=0}^{n-m}\sum_{i_{m-1}=0}^{i_m}(-1)^{i_m+i_{m-1}}d^{i_m+i_{m-1}}\int_{Y_{m-2}}c_1(L_m)_{\mid Y_{m-2}}^{i_m+1}\wedge c_1(L_{m-2})_{\mid Y_{m-1}}^{i_{m-1}+1}\wedge c_{n-m-i_m-i_{m-1}}(Y_{m-2}).
\end{equation}
Continuing by induction, we find that the Euler characteristic $\chi(Y_m)$ of $Y_m$ is equal to 
\begin{multline}\label{Adjunction4}
\sum_{i_m=0}^{n-m}\sum_{i_{m-1}=0}^{i_m}\cdots\sum_{i_1=1}^{i_2}(-1)^{i_m+i_{m-1}+\dots+i_1}d^{i_m+i_{m-1}+\dots+i_1}\times \\
\times\int_{X}c_1(L_m)^{i_m+1}\wedge c_1(L_{m-2})^{i_{m-1}+1}\wedge\cdots\wedge c_1(L_m)^{i_m+1}\wedge c_{n-m-i_m-i_{m-1}-\dots-i_1}(X)\\
=(-1)^{n-m}d^n\displaystyle\sum_{\substack{i_1+\dots+i_m=n-m \\ i_j\geq 0,\hspace{1mm} j\in\{1,\dots,m\}}}\int_Xc_1(L_1)^{i_1+1}\wedge\cdots\wedge c_1(L_m)^{i_m+1}+O(d^{n-1}),
\end{multline}
where in the last equality we used that the dominant term as $d\rightarrow\infty$ is given by the indices $(i_1,\dots,i_m)$ such that $i_1+\dots+i_m=n-m$. Recalling that $Y_m=Z_1\cap\cdots\cap Z_m$, we have the result.
\end{proof}
\begin{prop}\label{totalBetti} Let $L_1,\dots,L_m$ be ample line bundles over a complex projective  variety $X$ of dimension $n$. Let $s_1,\dots,s_m$ be generic holomorphic sections of $L^d_1,\dots,L^d_m$ and denote by $Z_{s_1},\dots,Z_{s_m}$ their vanishing loci. Then, as $d\rightarrow\infty$, we have the following asymptotic for the total Betti number of $Z_{s_1}\cap\dots\cap Z_{s_m}$:
$$b_*(Z_{s_1}\cap\cdots\cap Z_{s_m})=\mathrm{v}(L_1,\dots,L_m)d^n+O(d^{n-1})$$
where 
$$\mathrm{v}(L_1,\dots,L_m)=\displaystyle\sum_{\substack{i_1+\dots+i_m=n-m \\ i_j\geq 0,\hspace{1mm} j\in\{1,\dots,m\}}}\int_Xc_1(L_1)^{i_1+1}\wedge\cdots\wedge c_1(L_m)^{i_m+1}.$$
\end{prop}
\begin{proof}
Remark that the restriction of $L_m$ to $Z_{s_1}\cap\cdots\cap Z_{s_{m-1}}$ is an ample line bundle and then, for $d$ large enough, it embeds  $Z_{s_1}\cap\cdots\cap Z_{s_{m-1}}$ into some complex projective space $\P^N$. In particular, $Z_{s_1}\cap\cdots\cap Z_{s_m}$  is obtained as the intersection of a generic hyperplane  $H$ of $\P^N$  with $Z_{s_1}\cap\cdots\cap Z_{s_{m-1}}$.

By Lefschetz hyperplane theorem, we have $b_i(Z_{s_1}\cap\cdots\cap Z_{s_m})=b_i(Z_{s_1}\cap\cdots\cap Z_{s_{m-1}})$ for any $i\in\{0,\dots,n-m-1\}$. By induction, we then obtain that $b_i(Z_{s_1}\cap\cdots\cap Z_{s_m})=b_i(X)$ for any $i\in\{0,\dots,n-m-1\}$. In particular, for any $i\in\{0,\dots,n-m-1\}$, the Betti number $b_i(Z_{s_1}\cap\cdots\cap Z_{s_m})$ does not depend on $d$. By Poincar\'e duality, we also have $b_i(Z_{s_1}\cap\cdots\cap Z_{s_m})=b_{2n-2m-i}(Z_{s_1}\cap\cdots\cap Z_{s_m})$, so that the only Betti number of $Z_{s_1}\cap\cdots\cap Z_{s_m}$ that depends on $d$ is $b_{n-m}(Z_{s_1}\cap\cdots\cap Z_{s_m})$. In particular $$b_*(Z_{s_1}\cap\cdots\cap Z_{s_m})=b_{n-m}(Z_{s_1}\cap\cdots\cap Z_{s_m})+O(1)=(-1)^{n-m}\chi(Z_{s_1}\cap\cdots\cap Z_{s_m})+O(1)$$ as $d\rightarrow\infty$. The result then follows from Proposition \ref{totalBetti}.
\end{proof}
\subsection{Lefschetz pencils and degree of the discriminant}\label{seclefanddeg} The main result of this section is the computation of the degree of the discriminant, see Lemma \ref{degree of discriminant}. This will use the estimates on the topology of a generic section $s\in H^0(X,\oplus_{i=1}^mL_i^d)\setminus \Delta_d$.

\begin{lemma}[Degree of the discriminant]\label{degree of discriminant} Let $L_1,\dots,L_m$ be a  ample line bundle over a complex projective variety $X$ of dimension $n$ and denote by $\Delta_d$  the discriminant  in $H^0(X,\oplus_{i=1}^{m}L_i^d)$. Then, there exists an homogeneous  polynomial $Q_d$ vanishing on $\Delta_d$ and such that $${\deg(Q_d)=r(L_1,\dots,L_m)d^n+O(d^{n-1})},$$
where $r(L_1,\dots,L_m)$ equals
\begin{multline}
\displaystyle\sum_{k=1}^m \bigg(\displaystyle\sum_{\substack{i_1+\dots+i_{k-1}+i_{k+1}\dots+i_{m}=n-m+1 \\ i_j\geq 0,\hspace{1mm} j\in\{1,\dots,m\}\setminus\{k\}}}\int_Xc_1(L_1)^{i_1+1}\wedge\cdots\wedge c_1(L_{m-1})^{i_{m-1}+1} \\ +\displaystyle\sum_{\substack{i_1+\dots+i_m=n-m \\ i_j\geq 0,\hspace{1mm} j\in\{1,\dots,m\}}}\int_Xc_1(L_1)^{i_1+1}\wedge\cdots\wedge c_1(L_m)^{i_m+1} \\ +\displaystyle\sum_{\substack{i_1+\dots+i_m+i_{m+1}=n-m-1\\ i_j\geq 0,\hspace{1mm} j\in\{1,\dots,m+1\}}}\int_Xc_1(L_1)^{i_1+1}\wedge\cdots\wedge c_1(L_k)^{i_k+i_{m+1}+2}\wedge\cdots\wedge c_1(L_m)^{i_m+1}\bigg).
\end{multline}
\end{lemma}
\begin{proof}  First, we remark that if $(s_1,\dots,s_m)\in \Delta_d$ then $(\lambda_1s_1,\dots,\lambda_ms_m)\in \Delta_d$   for any $\lambda_1,\dots,\lambda_m\in\C^*$. This implies that  the degree of $\Delta_d$  equals the number of intersection points of $\P\Delta_d\subset \oplus_{i=1}^m\P H^0(X,L_i^d)$ with a generic line $\gamma\subset\oplus_{i=1}^m\P H^0(X,L_i^d)$, whose homology class $[\gamma]$ is the class of multidegree $(1,\dots,1)\in H_2(\oplus_{i=1}^m\P H^0(X,L_i^d),\mathbf{Z})\simeq \oplus_{i=1}^mH_2(\P H^0(X,L_i^d),\mathbf{Z})$. This number equals the cap product  between the fundamental class   $[\P\Delta_d]$ of $\P\Delta_d$ and the class $(1,\dots,1)\in \oplus_{i=1}^mH_2(\P H^0(X,L_i^d),\mathbf{Z})$. Denoting by $\cap$ the cap product, we have the equality  
\begin{equation}\label{cap product}
\deg(\Delta_d)=[\P\Delta_d]\cap (1,\dots,1)=[\P\Delta_d]\cap (1,0,\dots,0)+[\P\Delta_d]\cap (0,1,\dots,0)+[\P\Delta_d]\cap (0,0,\dots,1).
\end{equation}
 Let us then compute each of these cap products separately, starting, for simplicity, with $[\P\Delta_d]\cap (0,0,\dots,1)$. In order to do this, let us
 consider $([s_1],\dots,[s_{m-1}])\in\oplus_{i=1}^{m-1} \P H^0(X,L_i^d)$ such that $Z_{s_1}\cap\dots\cap Z_{s_{m-1}}$ is a smooth complete intersection of codimension $m-1$ in $X$. We can then choose a  generic pair of sections $[s_m],[s_m']\in \P H^0(X,L_m^d)$ such that   the line $\gamma:=\lambda\big([s_1],\dots,[s_{m-1}],[s_m]\big)+\mu\big([s_1],\dots,[s_{m-1}], [s_m']\big)$, with $[\lambda:\mu]\in \P^1$, intersects transversally $\P\Delta_d$. We have  that the fundamental class of the line $[\gamma]$ is the class $(0,\dots,0,1)\in \oplus_{i=1}^mH_2(\P H^0(X,L_i^d),\mathbf{Z})$.
 
  Remark now that the map
 $u:Z_{s_1}\cap\dots\cap Z_{s_{m-1}}\dashrightarrow \P^{1}$ defined by $u(x)=[s_m(x):s_m'(x)]$ is a Lefschetz pencil on $Z_{s_1}\cap\dots\cap Z_{s_{m-1}}$ and  points in the intersection $\gamma\cap\P\Delta_d$ correspond to critical points of $u$.
 By \cite[Equation (1)]{anc}, the number of critical points $\#\crit(u)$ of $u$ equals $$
(-1)^{n-m+1}\chi(Z_{s_1}\cap\dots\cap Z_{s_{m-1}})+(-2)^{n-m}\chi(Z_{s_1}\cap\dots\cap  Z_{s_{m}})+(-1)^{n-m+1}\chi(Z_{s_1}\cap\dots\cap Z_{s_{m}}\cap Z_{s'_{m}}).
$$
By Proposition \ref{euler}, the latter equals $r_md^n+O(d^{n-1})$ where
\begin{multline}
r_m= \displaystyle\sum_{\substack{i_1+\dots+i_{m-1}=n-m+1 \\ i_j\geq 0,\hspace{1mm} j\in\{1,\dots,m-1\}}}\int_Xc_1(L_1)^{i_1+1}\wedge\cdots\wedge c_1(L_{m-1})^{i_{m-1}+1} \\ +\displaystyle\sum_{\substack{i_1+\dots+i_m=n-m \\ i_j\geq 0,\hspace{1mm} j\in\{1,\dots,m\}}}\int_Xc_1(L_1)^{i_1+1}\wedge\cdots\wedge c_1(L_m)^{i_m+1} \\ +\displaystyle\sum_{\substack{i_1+\dots+i_m+i_{m+1}=n-m-1\\ i_j\geq 0,\hspace{1mm} j\in\{1,\dots,m+1\}}}\int_Xc_1(L_1)^{i_1+1}\wedge\cdots\wedge c_1(L_m)^{i_m+i_{m+1}+2}.
\end{multline}
We then obtain $[\P\Delta_d]\cap (0,\dots,0,1)=r_md^n+O(d^{n-1})$.

Similarly, we have for any $k\in\{1,\dots,m\}$
\begin{equation}\label{critical2}
[\P\Delta_d]\cap (0,\dots,\underset{k-\textrm{th place}}{1},\dots,0)=r_kd^{n}+O(d^{n-1})
\end{equation}
with $r_k$ equal to
\begin{multline}\label{crit4}
r_k= \displaystyle\sum_{\substack{i_1+\dots+i_{k-1}+i_{k+1}\dots+i_{m}=n-m+1 \\ i_j\geq 0,\hspace{1mm} j\in\{1,\dots,m\}\setminus\{k\}}}\int_Xc_1(L_1)^{i_1+1}\wedge\cdots\wedge c_1(L_{m-1})^{i_{m-1}+1} \\ +\displaystyle\sum_{\substack{i_1+\dots+i_m=n-m \\ i_j\geq 0,\hspace{1mm} j\in\{1,\dots,m\}}}\int_Xc_1(L_1)^{i_1+1}\wedge\cdots\wedge c_1(L_m)^{i_m+1} \\ +\displaystyle\sum_{\substack{i_1+\dots+i_m+i_{m+1}=n-m-1\\ i_j\geq 0,\hspace{1mm} j\in\{1,\dots,m+1\}}}\int_Xc_1(L_1)^{i_1+1}\wedge\cdots\wedge c_1(L_k)^{i_k+i_{m+1}+2}\wedge\cdots\wedge c_1(L_m)^{i_m+1}.
\end{multline}
By summing Equation \eqref{critical2} over $k\in\{1,\dots,m\}$ we obtain 
\begin{equation}\label{critical3}
[\P\Delta_d]\cap (1,\dots,1)=(r_1+\dots+r_m)d^{n}+O(d^{n-1}).
\end{equation}
 The result then follows from Equations \eqref{cap product}, \eqref{critical3} and \eqref{crit4}.
\end{proof}

\section{Distance to the real discriminant and $1$-jet of sections}\label{secdistanceto}
Let us denote by $\R\Delta_d\subset \R H^0(X,\oplus_{i=1}^mL_i^d)$  the real locus of the discriminant, that is the subset of real sections $s=(s_1,\dots,s_m)$ which do not vanish transversally along $X$.
\begin{defn}\label{realdiscriminant} We denote $\Sigma_d\subset \R\Delta_d$  the space of real sections $s=(s_1,\dots,s_m)$ of $\oplus_{i=1}^mL^d_i$ which do not vanish transversally along $\R X$. We call $\Sigma_d$ the \textit{real discriminant}.
\end{defn}
   In this section, we estimate the function "distance to the real discriminant". This is the content of Lemma \ref{distance to the discriminant}, which is the main result of the section. In order to do this, in Section \ref{sec eval} we introduce the peak sections associated with $\oplus_{i=1}^mL_i^d$ and study their $\mathcal{C}^1$-norm.

\subsection{Evaluation and $1$-jet maps}\label{sec eval}   
Here, we introduce the so-called peak sections at real points of a real algebraic variety (see, for example,  \cite{tian,hor,gw1}).
This sections will be used to estimate the distance to the real discriminant.

Throughout this Section \ref{sec eval}, we consider a real ample holomorphic line bundle $L$ on $X$ equipped with a Hermitian metric $h$ with positive curvature $\omega$. This induces a scalar product on $\R H^0(X,L^d)$ defined by $$\langle s_1,s_2\rangle_{\mathcal{L}^2}=\int_X h^d(s_1,s_2)\frac{\omega^n}{n!}$$
for any $s_1,s_2\in \R H^0(X,L^d)$.
\begin{defn}[Evaluation maps]\label{defn eval} For any  $x\in \R X$, let $\R H_x$ be the kernel of the evaluation map  $${ev_x:s\in   \R H^0(X,L^d)\mapsto s(x)\in \R L^d_x}.$$ 
Similarly,  for any real tangent vector $v\in T_x^*\R X$ at $x$ we define the   map
$$ev_{2x,v}:s\in \R H_x \mapsto \nabla_v s(x)\in \R L^d_x,$$ where $\nabla$ is any connection on $L^d$ (indeed, if $s\in \R H_x$, then the value $\nabla s(x)$ does not depend on $\nabla$).
\end{defn}
\begin{defn}[Peak sections]\label{peak}  Let $x$ be a point in $\R X$.
A  \emph{peak section at $x$} is a generator $s_x$ of $(\ker ev_x)^{\perp}$ of unit $\mathcal{L}^2$-norm.
\end{defn}
\begin{defn}[First order peak sections]\label{ort basis of peak} Let $x$ be a points in $\R X$. \begin{itemize}
\item For any $v\in T_x\R X$,  a \emph{first order peak section at $x$ associated with the tangent vector $v$} is a generator $\tilde{s}_{v}$ of $(\ker ev_{2x,v})^{\perp}$ of unit $\mathcal{L}^2$-norm.
\item Let $\mathcal{B}=\{v_1,\dots,v_n\}$ be an orthonormal basis of $T_x\R X$. Let $\tilde{s}_{v_1},\dots,\tilde{s}_{v_n}$ be a first order peak sections at $x$ associated with the tangent vectors $v_1,\dots,v_n$. We call the \emph{first order peak sections at $x$ associated with the basis $\mathcal{B}$ } the sections of the orthonormal family $\{s_{v_1},\dots,s_{v_n}\}$ obtained by applying the Gram-Schmidt process to the family $\{\tilde{s}_{v_1},\dots,\tilde{s}_{v_n}\}$.
 \end{itemize}
\end{defn}
The next two lemmas estimate the pointwise norm of the peak sections and of their derivatives. These estimates are nowadays standard and are essentially proved in \cite{tian}. Since we use slightly different conventions, we will give a proof for the sake of completeness.
\begin{lemma}[Estimates of peak sections]\label{estimates of peak sections} We have the following uniform estimates for $x\in\R X$:
$\abs{s_x(x)}^2_{h^d}=\frac{d^n}{\pi^n}(1+O(d^{-1}))$ and $\abs{\nabla s_{x}(x)}^2_{h^d}=O(d^{n-1})$ , as $d\rightarrow\infty$. Here, we have denoted by $n$  the dimension of $X$, by $\nabla$  the Chern connection of $(L^d,h^d)$ and by $\abs{\cdot}_{h^d}$  the norm induced by $h^d$.
\end{lemma}
\begin{proof}
Let $B_d(z_1,z_2)$ be the Bergman kernel associated with the Hermitian line bundle $(L^d,h^d)$, then we have $\abs{s_x(x)}^2_{h^d}=B_d(x,x)$ and $\abs{\nabla s_{x}(x)}^2_{h^d}=\frac{\abs{\nabla_1B_d(x,x)}_{h^d}^2}{B_d(x,x)}$, where $\nabla_1B_d(x,x)$ stands for the covariant derivative with respect to the first variable. To prove these identities, complete $\{s_x\}$ to a real orthonormal basis $\{s_x,s_1\dots,s_{N_d}\}$ of $\R H^0(X,L^d)$. Remark that $\{s_1\dots,s_{N_d}\}$ is an orthonormal basis of $\ker ev_x$. In particular, we have  $$B_d(x,x)=\abs{s_x(x)}^2_{h_d}+\sum_{i=1}^{N_d}\abs{s_i(x)}^2_{h_d}=\abs{s_x(x)}^2_{h_d},$$ where the first equality comes from the  definition of the Bergman kernel and the second follows from the fact that $s_i\in \ker ev_x$. The equality $\abs{s_x(x)}^2_{h^d}=\frac{d^n}{\pi^n}(1+O(d^{-1}))$ then follows from the fact $B_d(x,x)=\frac{d^n}{\pi^n}(1+O(d^{-1}))$, see, for example, \cite{ma2,zel}. 

For the equality $\abs{\nabla s_{x}(x)}^2_{h^d}=\frac{\abs{\nabla_1B_d(x,x)}_{h^d}^2}{B_d(x,x)}$, we first  differentiate with respect to the first variable the identity  $B_d(z,w)=s_x(z)\otimes s^*_x(w)+\sum_{i=1}^{N_d}s_i(z)\otimes s^*_i(w)$ and then we evaluate at $(z,w)=(x,x)$ and we obtain $$\nabla_1B_d(x,x)=\nabla s_x(z)\otimes s^*_x(x)+\sum_{i=1}^{N_d}\nabla s_i(x)\otimes s^*_i(x)=\nabla s_x(z)\otimes s^*_x(x),$$ where the second equality follows again from the fact that  $s_i\in \ker ev_x$. We then obtain $\abs{\nabla s_x(x)}^2_{h^d}=\frac{\abs{\nabla_1B_d(x,x)}_{h^d}^2}{\abs{ s_x(x)}^2_{h^d}}$. The result then follows from the fact that $\abs{\nabla_1B_d(x,x)}_{h^d}=O(d^{n-1/2})$, see \cite[Theorem 4.2.1]{ma2}
\end{proof}
\begin{lemma}[Estimates of first order peak sections]\label{estimates of first peak sections} We have the following uniform estimates for any $x\in\R X$ and any $\{v_1,\dots,v_n\}$ orthonormal basis of $T_x\R X$:
 $$\abs{\nabla_{v_i}s_{v_i}(x)}^2_{h^d}=\frac{d^{n+1}}{\pi^n}(1+O(d^{-1}))\hspace{3mm} \textrm{and}\hspace{3mm} \abs{\nabla_{v_j}s_{v_i}(x)}^2_{h^d}=O(d^n)\hspace{2mm} \textrm{for}\hspace{2mm} i\neq j$$  as $d\rightarrow\infty$. Here, we have denoted by $n$  the dimension of $X$, by $\nabla$  the Chern connection of $(L^d,h^d)$ and by $\abs{\cdot}_{h^d}$  the norm induced by $h^d$.
\end{lemma}
\begin{proof}
Let $\tilde{s}_{v_i}$ be a first order peak sections at $x$ associated with the tangent vector $v_i$ (see Definition \ref{ort basis of peak}). Following the lines of the proof of Lemma \ref{estimates of peak sections}, we obtain that  $\abs{\nabla_{v_i}\tilde{s}_{v_i}(x)}^2_{h^d}=\frac{d^{n+1}}{\pi^n}(1+O(d^{-1}))$ and $\abs{\nabla_{v_j}\tilde{s}_{v_i}(x)}^2_{h^d}=O(d^n)$ for $i\neq j$

Remark that   $s_{v_1}=\tilde{s}_{v_1}$ and so we already have the result for $s_{v_1}$. Let us suppose that we have the proved the desired estimates for the sections $s_{v_1},\dots,s_{v_{n-1}}$ and let us prove them for $s_{v_n}$.  By construction, the section $s_{v_n}$ is constructed from the sections $\{s_{v_1},\dots,s_{v_{n-1}}\}$ as follows:
\begin{equation}\label{gramschmidtprocess}
 s_{v_n}=\frac{\tilde{s}_{v_n}-\displaystyle\sum_{i=1}^{n-1}\langle\tilde{s}_{v_n},s_{v_i}\rangle_{\mathcal{L}^2}s_{v_i}}{\norm{\tilde{s}_{v_n}-\displaystyle\sum_{i=1}^{n-1}\langle\tilde{s}_{v_n},s_{v_i}\rangle_{\mathcal{L}^2}s_{v_i}}_{\mathcal{L}^2}}.
 \end{equation}
 By \cite[Lemma 3.1]{tian}, the $\mathcal{L}^2$-scalar product between $\tilde{s}_{v_i}$ and $\tilde{s}_{v_j}$ is $O(d^{-1})$ for any $i\neq j$. Following the Gram-Schmidt process, this implies  $\langle\tilde{s}_{v_n},s_{v_i}\rangle_{\mathcal{L}^2}=O(d^{-1})$. In particular, we obtain $\norm{\tilde{s}_{v_n}-\displaystyle\sum_{i=1}^{n-1}\langle\tilde{s}_{v_n},s_{v_i}\rangle_{\mathcal{L}^2}s_{v_i}}_{\mathcal{L}^2}=1+O(d^{-1})$ which gives us 
 \begin{equation}\label{gram2}
 \eqref{gramschmidtprocess}=\bigg(\tilde{s}_{v_n}-\displaystyle\sum_{i=1}^{n-1}a_is_{v_i}\bigg)(1+O(d^{-1}))
 \end{equation}
 where $a_i=O(d^{-1})$. From \eqref{gram2}, the induction hypothesis and the properties of $\tilde{s}_{v_n}$, we obtain $\abs{\nabla_{v_n}s_{v_n}(x)}^2_{h^d}=\frac{d^{n+1}}{\pi^n}(1+O(d^{-1}))$ and  $\abs{\nabla_{v_i}s_{v_n}(x)}^2_{h^d}=O(d^n)$ for $i\neq n$, hence the result.
\end{proof}
\subsection{Distance to the real discriminant}\label{sec distance}
Let us now consider $m$ real ample line bundles $L_1,\dots,L_m$ on $X$. We equip each line bundle
 $L_i$ with a real Hermitian metric $h_i$ with positive curvature. This induces a $\mathcal{L}^2$-scalar product on $\R H^0(X,\oplus_{i=1}^mL_i^d)$  defined by Equation \eqref{l2 scalar}. Given $s=(s_1,\dots,s_m)\in \R H^0(X,\oplus_{i=1}^mL_i^d)$, the next lemma estimates the $\mathcal{L}^2$-distance  from $s$ to the real discriminant $\Sigma_d$, see Definition \ref{realdiscriminant}. This distance is computed with respect to the $\mathcal{L}^2$-scalar product, that is   $$\mathrm{dist}_{\Sigma_d}(s):=\min_{s'\in \Sigma_d}\norm{s-s'}_{\mathcal{L}^2}.$$
The estimate of $\mathrm{dist}_{\Sigma_d}(s)$ in Lemma \ref{distance to the discriminant} is given in term of the distance induced by the $\mathcal{C}^{1}(\R X)$-norm.
\begin{defn}[$\mathcal{C}^1(\R X)$-norm]\label{defC1} We define the $\mathcal{C}^1(\R X)$-norm of a real holomorphic global section $s=(s_1,\dots,s_m)$ of $L_1^d\oplus\dots\oplus L_m^d$ to be $$\norm{s}_{\mathcal{C}^1(\R X)}=\max_{x\in \R X}\big(\abs{s(x)}^2_{h_d}+\abs{\nabla s(x)}^2_{h_d}\big)^{1/2}.$$ Here we use the following notations:
\begin{itemize}
 \item $\abs{s(x)}^2_{h_d}=\sum_{i=1}^m\abs{s_i(x)}^2_{h_i^d}$ is the norm induced by the Hermitian metrics $h_i^d$ on $L_i^d$, $i\in\{1,\dots,m\}$.
 \item   $\abs{\nabla s(x)}_{h_d}^2=\displaystyle\sum_{i=1}^m\sum_{j=1}^n\abs{\nabla_{v^{(i)}_j}s(x)}_{h^d}^2$, where $\nabla$ is the Chern connection of $(L_i,h_i)$ and where $\big\{v^{(i)}_1,\dots, v^{(i)}_n\big\}$ is an orthonormal basis of $T_x\R X$ with respect to the Riemannian metric induced by the curvature form of $(L_i,h_i)$.
 \end{itemize}
\end{defn}
\begin{lemma}[Distance to the discriminant]\label{distance to the discriminant} Let $(L_1,c_{L_1}),\dots,(L_m,c_{L_m})$ be  real  ample line bundles over a real algebraic variety $(X,c_X)$ of dimension $n$.  Then, there exists  $d_0\in\mathbb{N}$ such that, for any $d\geq d_0$ and any $s\in \R H^0(X,\oplus_{i=1}^mL_i^d)$, we have  $$\mathrm{dist}_{\Sigma_d}(s)\leq \mathrm{dist}_{\mathcal{C}^{1}(\R X)}(s,\Sigma_d).$$
\end{lemma} 
\begin{proof}
For any $x\in\R X$, let us denote by $\Sigma_{d,x}$ the space of sections $s\in  \R H^0(X,\oplus_{i=1}^mL^d_i)$ that do not vanishing transversally at $x$, that is 
$$\Sigma_{d,x}:=\big\{s\in\R H^0(X,\oplus_{i=1}^mL^d_i), s(x)=0\hspace*{1.5mm} \textrm{and}\hspace*{1.5mm}  \nabla s(x) \hspace*{1.5mm}  \textrm{is not surjective} \big\}.$$
We have $\bigcup_{x\in \R X}\Sigma_{d,x}=\Sigma_{d}$, so that  
\begin{equation}\label{minL2}
\mathrm{dist}_{\Sigma_d}(s)=\min_{x\in \R X}\min_{s'\in\Sigma_{d,x}}\norm{s-s'}_{\mathcal{L}^2}. 
\end{equation}
Let us consider the real $1$-jets space of $\oplus_{i=1}^mL^d$ at $x$ $$J_x^1\big(\oplus_{i=1}^m\R L^d_i\big):=\bigoplus_{i=1}^m(\R L^d_i)_x\oplus\bigoplus_{i=1}^m(T^*\R X\otimes\R L^d_i)_x$$ and define $W_x\subset J_x^1\big(\bigoplus_{i=1}^m\R L^d_i\big)$ to be the image of $\Sigma_{d,x}$ under the $1$-jet map $$j^1_x:\R H^0(X,\oplus_{i=1}^m L^d_i)\rightarrow J_x^1\big(\bigoplus_{i=1}^m\R L^d_i\big)$$ which maps $s$ to $j^1_x(s):=(s(x),\nabla s(x))$.

 The vector bundle $J^1\big(\bigoplus_{i=1}^m\R L^d_i\big)$ over $\R X$ is naturally equipped with a metric (induced by the metric $h_d=h^d_1\oplus\dots\oplus h^d_m$ on $\oplus_{i=1}^mL^d_i$), and, with respect to this metric,  the distance in the fiber $J_x^1\big(\bigoplus_{i=1}^m\R L^d_i\big)$ between $j^1_x(s)$ and $W_x$  equals $$\mathrm{dist}_{h_d}(j^1_x(s),W_x):=\min_{s'\in \R \Delta_{d,x}}\big(\abs{s(x)-s'(x)}^2_{h_d}+\abs{\nabla s(x)-\nabla s'(x)}^2_{h_d}\big)^{1/2}.$$  Remark that we have the following inequality 
 \begin{multline*}\mathrm{dist}_{h_d}(j^1_x(s),W_x)  \leq \min_{s'\in  \Sigma_{d,x}}\max_{y\in \R X}\big(\abs{s(y)-s'(y)}^2_{h_d}+\abs{\nabla s(y)-\nabla s'(y)}^2_{h_d}\big)^{1/2}
\\ =:\mathrm{dist}_{\mathcal{C}^1(\R X)}(s,\Sigma_{d,x}),
 \end{multline*} 
 which implies $$\min_{x\in \R X}\mathrm{dist}_{h_d}(j^1_x(s),W_x)\leq \mathrm{dist}_{\mathcal{C}^1(\R X)}(s,\Sigma_{d}).$$
In particular, the last inequality says that, in order to prove the lemma, it is enough to prove the inequality $\mathrm{dist}_{\Sigma_d}(s)\leq \min_{x\in \R X}\mathrm{dist}_{h_d}(j^1_x(s),W_x)$ which, in turn, is implied by the inequality 
\begin{equation}\label{ineqOfDistance}
\mathrm{dist}_{\Sigma_{d,x}}(s)\leq \mathrm{dist}_{h_d}(j^1_x(s),W_x), \hspace{1mm}\textrm{for any}\hspace{1mm} x\in \R X.
\end{equation}
In the remaining part of the proof, we will prove Equation \eqref{ineqOfDistance}. For any $i\in\{1,\dots, m\}$, let $\{v_1^{(i)},\dots,v_m^{(i)}\}$ be an orthonormal basis of $T_x\R X$ with the respect to the Riemannian metric induced by the (positive) curvature form of the Hermitian metric $h_i$.

For any $i\in\{1,\dots, m\}$ let $s_0^{(i)}\in  \R H^0(X,L_i^d)$ be a peak section at $x$ (see Definition \ref{peak}). Similarly, let $s_{1}^{(i)},\dots, s_{n}^{(i)}\in \R H^0(X,L_i^d)$ be the orthonormal family of the first order peak sections at $x$ associated with the basis $\{v_1^{(i)}\dots,v_m^{(i)}\}$ (see Definition \ref{ort basis of peak}). We can then write
\begin{equation}\label{basisofs}
s=\big(\sum_{j=0}^na_j^{(1)}s_j^{(1)},\dots,\sum_{j=0}^na_j^{(m)}s_j^{(m)}\big)+ \big(\tau_1\dots, \tau_m\big)
\end{equation}
where $\tau_i(x)=\nabla\tau_i(x)=0$ for any $i\in\{1,\dots,m\}$.
Similarly,   for any section $s'\in \R H^0(X,\oplus_{i=1}^mL_i^d)$, we can write
\begin{equation}\label{basisofs1}
s'=\big(\sum_{j=0}^nb_j^{(1)}s_j^{(1)},\dots,\sum_{j=0}^nb_j^{(m)}s_j^{(m)}\big)+ \big(\tau'_1\dots, \tau'_m\big)
\end{equation}
where $\tau'_i(x)=\nabla\tau'_i(x)=0$ for any $i\in\{1,\dots,m\}$. With this notations, we have
\begin{equation}\label{diffOfL2}
\norm{s-s'}^2_{\mathcal{L}^2}=\sum_{j=0}^n\sum_{i=1}^m\big(a_j^{(i)}-b_j^{(i)}\big)^2+\sum_{i=1}^m\norm{\tau_i-\tau'_i}^2_{\mathcal{L}^2}.
\end{equation}
Now, if $s'\in\Sigma_{d,x}$, then we have $b_0^{(i)}=0$ for any $i\in\{1,\dots,m\}$ and the $n\times m$ real matrix $(b_j^{(i)})_{i,j}$ has rank smaller or equal than $m-1$ (this reflects the condition "$\nabla s'(x)$ is not surjective"). Moreover, as we want to minimize the $\mathcal{L}^2$-distance between $s$ and $s'$,   we can choose $\tau_i'=\tau_i$, for any $i\in\{1,\dots,m\}$ (indeed, remark that the section $\tau'=(\tau_1',\dots,\tau_m')$ vanishes at $x$ with order at least $2$, in particular for any such $\tau'$ and any $s'\in\Sigma_{d,x}$, we have $s'+\tau'\in\Sigma_{d,x}$). This implies that 
\begin{equation}\label{minOfL2}
\mathrm{dist}^2_{\Sigma_{d,x}}(s)=\min_{s'\in\Sigma_{d,x}}\norm{s-s'}^2_{\mathcal{L}^2}=\sum_{i=1}^m\big(a_0^{(i)}\big)^2+\min_{(b_j^{(i)})_{i,j}\in \textrm{Sing}_{n\times m}}\sum_{j=1}^n\sum_{i=1}^m\big(a_j^{(i)}-b_j^{(i)}\big)^2
\end{equation}
where $\textrm{Sing}_{n\times m}$ denotes the space of $n\times m$ real matrices of rank smaller or equal than $m-1$.

Let us now compute $\mathrm{dist}^2_{h_d}(j^1_x(s),W_x)$. We keep the notations \eqref{basisofs} and \eqref{basisofs1}. By Lemmas \ref{estimates of peak sections} and \ref{estimates of first peak sections},  we have, for any $s'\in\R H^0(X,\oplus_{i=1}^mL_i^d)$, 
\begin{equation}\label{diffOfC1}
\abs{j^1_x(s)-j^1_x(s')}^2_{h_d}=\sum_{i=1}^m\big(a_0^{(i)}-b_0^{(i)}\big)^2\frac{d^n}{\pi^n}\big(1+O(d^{-1})\big)+\sum_{j=1}^n\sum_{i=1}^m\big(a_j^{(i)}-b_j^{(i)}\big)^2\frac{d^{n+1}}{\pi^n}\big(1+O(d^{-1})\big).
\end{equation}
Now, if $s'\in\Sigma_{d,x}$, we obtain that $b_0^{(i)}=0$ for any $i\in\{1,\dots,m\}$ and the $n\times m$ real matrix $(b_j^{(i)})_{i,j}$ lies in $\textrm{Sing}_{n\times m}$. In particular, by Equation \eqref{diffOfC1}, we obtain 
\begin{multline}\label{minOfC1}
\textrm{dist}^2_{h_d}(j^1_x(s),W_x)=\min_{s'\in\Sigma_{d,x}}\abs{j^1_x(s)-j^1_x(s')}^2_{h_d}\\
=\sum_{i=1}^m\big(a_0^{(i)}\big)^2\big(\frac{d^n}{\pi^n}+O(d^{n-1})\big)+\min_{(b_j^{(i)})_{i,j}\in \textrm{Sing}_{n\times m}}\sum_{j=1}^n\sum_{i=1}^m\big(a_j^{(i)}-b_j^{(i)}\big)^2\big(\frac{d^{n+1}}{\pi^n}+O(d^{n})\big)\\
=d^n\bigg(\sum_{i=1}^m\big(a_0^{(i)}\big)^2+\min_{(b_j^{(i)})_{i,j}\in \textrm{Sing}_{n\times m}}\sum_{j=1}^n\sum_{i=1}^md\big(a_j^{(i)}-b_j^{(i)}\big)^2\bigg)\big(\pi^{-n}+O(d^{-1})\big).
\end{multline}
In particular, for $d$ large enough, the quantity appearing in Equation \eqref{minOfC1} is bigger than the one in Equation \eqref{minOfL2}. This is exactly the inequality \eqref{ineqOfDistance}, which proves the result.
\end{proof}
\section{An orthogonal decomposition for real global sections}\label{secorthogonaldecomposition}
In this section we define an orthogonal decomposition of the space of real holomorphic sections of $\oplus_{i=1}^m{L^d_i}$ (see Notation \ref{orthogonal component}) which will be a key ingredient for the proof of Theorem \ref{theorem approximation}. In order to do this, first recall  the following result from \cite{anc5}.
\begin{prop}\cite[Proposition 2.1]{anc5}\label{existance of sigma} Let $L$ be an ample real holomorphic line bundle over a real algebraic variety $X$. There exists an even positive integer $k_0$ such that for any even $k\geq k_0$ there exists a real  section $\sigma$ of $L^{k}$ with the following properties: (i)  $\sigma$ vanishes transversally and (ii) $\R Z_{\sigma}$ is empty.
\end{prop}
\begin{conv}\label{notation of sigma} For any $i\in\{1,\dots,m\}$ and any  even integer $k$ large enough, we  denote by $\sigma_i$ a real global section of $L_i^{k}$ satisfying the properties of Proposition \ref{existance of sigma}.
\end{conv}
 \begin{defn}\label{vanishing order} Let $\sigma_i\in  \R H^0(X,L_i^k)$ be a section given by Notation \ref{notation of sigma}, for some fixed even integer $k$ large enough and denote by $\sigma=(\sigma_1,\dots,\sigma_m)\in  \R H^0(X,\oplus_{i=1}^mL_i^k)$.
 For any pair of integers $d$ and $\ell$,  we define the subspace $\R H_{d,\sigma^\ell}$ of $\R H^0(X,\oplus_{i=1}^mL_i^d)$ to be  the space of sections $s=(s_1,\dots,s_m)\in  \R H^0(X,\oplus_{i=1}^mL_i^d)$ such that $s_i=\sigma_i^{\ell}\otimes s_i'$, for some $s_i'\in  \R H^0(X,L_i^{d-k\ell})$.
\end{defn}

\begin{conv}[Orthogonal decomposition]\label{orthogonal component}  
 For any real section $s\in \R H^0(X,\oplus_{i=1}^mL_i^d)$ there exists an unique  orthogonal decomposition $s=s_{\sigma^{\ell}}^{\perp}+s_{\sigma^{\ell}}^{0}$ with $s_{\sigma^{\ell}}^{0}\in \R H_{d,\sigma^\ell}$ and $s_{\sigma^{\ell}}^{\perp}\in \R H^{\perp}_{d,\sigma^\ell}$. (Here,  $\R H_{d,\sigma^\ell}$ is as in Definition \ref{vanishing order} and  its orthogonal is with respect to the $\mathcal{L}^2$-scalar product defined in Equation \eqref{l2 scalar}.) 
\end{conv}

\begin{prop}\label{partialbergman} There exists a positive real number $t_0$ such that, for any $t\in(0, t_0)$,  we have the uniform estimate $\norm{\tau}_{\mathcal{C}^1(\R X)}=O(d^{-\infty})$ for any real section $\tau\in\R H_{d,\sigma^{\lfloor t d\rfloor}}^{\perp}$ with $\norm{\tau}_{\mathcal{L}^2}=1$, 
as $d\rightarrow\infty$. (Here, $\R H_{d,\sigma^{\lfloor t d\rfloor}}$ is as in Definition \ref{vanishing order} and  $\lfloor t d \rfloor$ is the greatest integer less than or equal to $t d$.)
\end{prop}
\begin{proof}
Let $\tau=(\tau_1,\dots,\tau_m)\in \R H_{d,\sigma^{\lfloor t d\rfloor}}^{\perp}$ be such that $\norm{\tau}_{\mathcal{L}^2}=1$. This implies in particular that $\norm{\tau_i}_{\mathcal{L}^2}\leq 1$ for any $i\in\{1,\dots,m\}$. By \cite[Proposition 2.6]{anc5}, we have the uniform estimate $\norm{\tau}_{\mathcal{C}^1(\R X)}=O(d^{-\infty})$. The result then follows from $\norm{\tau}_{\mathcal{C}^1(\R X)}\leq \big(\sum_{i=1}^m\norm{\tau_i}^2_{\mathcal{C}^1(\R X)}\big)^{1/2}$
\end{proof}
\begin{prop}\label{logbergman} Let $k$  be an integer large enough. There exists $c>0$ (depending on $k$) such that we have the uniform estimate $\norm{\tau}_{\mathcal{C}^1(\R X)}\leq O(e^{-c\sqrt{d}\log d})$ for any real section $\tau\in \R H_{d,\sigma}^{\perp}$ with $\norm{\tau}_{\mathcal{L}^2}=1$, as $d\rightarrow\infty$.  If the real Hermitian metrics $h_i$ on $L_i$ are  analytic, then  we have the uniform estimate $\norm{\tau}_{\mathcal{C}^1(\R X)}\leq O(e^{-c d})$ for any real section $\tau\in \R H_{d,\sigma}^{\perp}$ with $\norm{\tau}_{\mathcal{L}^2}=1$.
\end{prop}
\begin{proof}
The proof follows the lines of the proof of Proposition \ref{vanishing order}, by using \cite[Proposition 2.7]{anc5} instead of \cite[Proposition 2.6]{anc5}.
\end{proof}
Using the notation of the orthogonal decomposition given in Notation \ref{orthogonal component}, Propositions \ref{partialbergman} and \ref{logbergman} in particular imply the following result.
 \begin{prop}\label{estimatepartial} Let $\sigma=(\sigma_1,\dots,\sigma_m)\in \R H^0(X,\oplus_{i=1}^mL_i^k)$ be a section given by Definition \ref{vanishing order}, for some fixed $k$ large enough. 
\begin{enumerate}
 \item There exists a positive real number $t_0$ such that for any $t\in(0,t_0)$ the following happens. Let $C>0$ and $r\in\mathbb{N}$. For any sequence $w_d$ with $w_d\geq Cd^{-r}$, there exists  $d_0\in\mathbb{N}$, such that for any $d\geq d_0$ and any  $s\in\R H^0(X,\oplus_{i=1}^mL_i^d)$ we have  $$\norm{s_{\sigma^{\lfloor t d\rfloor}}^{\perp}}_{\mathcal{C}^1(\R X)}< w_d\norm{s}_{\mathcal{L}^2}.$$
 \item There exist two positive constants  $c_1$ and $c_2$ such that, for any sequence of real numbers $w_d$ with $w_d\geq c_1e^{-c_2\sqrt{d}\log d}$ and any $s\in\R H^0(X,\oplus_{i=1}^mL_i^d)$ we have 
  $$ \norm{s_{\sigma}^{\perp}}_{\mathcal{C}^1(\R X)}< w_d\norm{s}_{\mathcal{L}^2}.$$
  If, moreover, the real Hermitian metrics $h_i$ on $L_i$ are analytic, then the last estimate is true for any sequence $w_d$ with $w_d\geq c_1e^{-c_2d}$.
\end{enumerate} 
Here, $s_{\sigma}^{\perp}$ and $s_{\sigma^{\lfloor t d\rfloor}}^{\perp}$ are given by Notation \ref{orthogonal component}. 
\end{prop}
\begin{proof}
Let us prove point \textit{(1)}. 
Take $s\in\R H^0(X,\oplus_{i=1}^mL_i^d)$ and consider the section $(\norm{s}_{\mathcal{L}^2})^{-1}s$, which has $\mathcal{L}^2$--norm equal to $1$.
Now, by Proposition \ref{partialbergman}, as $(\norm{s}_{\mathcal{L}^2})^{-1}\norm{s^{\perp}_{\sigma^{\lfloor t d\rfloor}}}_{\mathcal{L}^2}\leq 1$, there exists a  constant $c_r>0$ (not depending on $d$) such that $(\norm{s}_{\mathcal{L}^2})^{-1}\norm{s_{\sigma^{\lfloor t d\rfloor}}^{\perp}}_{\mathcal{C}^1(\R X)}\leq c_r d^{-r-1}$, which is strictly smaller than $w_d$, for $d$ large enough. This proves  point \textit{(1)} of the result.
 The proof of  point \textit{(2)} follows the same lines, using Proposition \ref{logbergman} instead of Proposition \ref{partialbergman}.
\end{proof}

\section{Proof of the main results}\label{SecProof}
The goal of this section is to prove  Theorems  \ref{theorem approximation} and \ref{theorem rarefaction}. 
Before this, let us recall some notations we have used so far. Let $L_1,\dots,L_m$ be real ample line bundles over a $n$--dimensional real algebraic variety $X$. We equip each line bundle  with a real Hermitian metric  with positive curvature. This induces a $\mathcal{L}^2$--scalar product on $ \R H^0(X,\oplus_{i=1}^mL_i^d)$ defined  in Equation \eqref{l2 scalar} and  a  Gaussian measure on $ \R H^0(X,\oplus_{i=1}^mL_i^d)$ denoted by $\mu_d$ and defined in Equation \eqref{gaussian measure}. 
Finally, we denote by $\Sigma_d$ the real discriminant (see Definition \ref{realdiscriminant}).
\subsection{Tubular neighborhoods of the real discriminant} 
 A \textit{tubular conical neighborhood} of the real discriminant $\Sigma_d$  in $\R H^0(X,\oplus_{i=1}^mL^d_i)$  is a tubular neighborhood of  $\Sigma_d$ in $\R H^0(X,\oplus_{i=1}^mL^d_i)$ which is also a cone (that is, if $s$ is in the neighborhood, then $\lambda s$ is also in the neighborhood, for any $\lambda\in\R^*$). 

  The next lemma  estimates the measure of small tubular conical neighborhoods of the real discriminant $\Sigma_d$. It is a special case of \cite[Theorem 21.1]{burcondition}  (see also \cite[Proposition 4]{diatta} and \cite[Lemma 3.4]{anc5}).
\begin{lemma}[Volume of tubular conical neighborhoods] \label{volume of tubular neighborhood} Let $L_1,\dots,L_m$ be  real Hermitian ample line bundles over a real algebraic variety $X$ of dimension $n$. Then there exists a positive constant $c$ (not depending on $d$), such that, for any sequence $r_d$ of positive real numbers verifying $r_d\leq cd^{-2n}$, one has 
$$\mu_d\{s\in \R H^0(X,\oplus_{i=1}^mL_i^d), \mathrm{dist}_{\Sigma_d}(s)\leq r_d\norm{s}_{\mathcal{L}^2}\}\leq O(r_dd^{2n}).$$
(Here, $\mu_d$ is the Gaussian probability measure defined in Equation \eqref{gaussian measure}.)
\end{lemma}
\begin{proof} Let us denote by $S_d$ the unit sphere in $\R H^0(X,\oplus_{i=1}^mL_i^d)$, that is $$S_d:=\{s\in \R H^0(X,\oplus_{i=1}^mL_i^d),\norm{s}_{\mathcal{L}^2}=1\}.$$ Let us also denote by  $S\Sigma_d$ the trace of the discriminant on $S_d$, that is $S\Sigma_d:=S_d\cap \Sigma_d$.
Now, if we denote by  $\nu_d$ the probability measure on $S_d$ induced by its volume form (that is, for any $U\subset S_d$, $\nu_d(U)=\Vol(U)\Vol(S_d)^{-1}$), then the Gaussian measure of every cone $C_d$ in $\R H^0(X,\oplus_{i=1}^mL_i^d)$ equals $\nu_d(C_d\cap S_d)$.
This implies that the Gaussian measure  of the cone $C_d=\{\mathrm{dist}_{\Sigma_d}(s)\leq r_d\norm{s}_{\mathcal{L}^2}\}$ we are interested in  equals $\nu_d\{s\in S_d, \mathrm{dist}(s,S\Sigma_d)\leq r_d\}.$ In order to obtain the result, it is then equivalent to prove the estimate \begin{equation}\label{numeasure0}\nu_d\{s\in S_d, \mathrm{dist}(s,S\Sigma_d)\leq r_d\}\leq O(r_dd^{2n}).
\end{equation}

Recall, that, by Lemma \ref{degree of discriminant}, there exists a polynomial of degree bounded by $cd^n$ (for some $c>0$ independent of $d$)  whose zero locus contains $S\Sigma_d$. We are then in the hypotheses of \cite[Theorem 21.1]{burcondition} which gives us the estimate
\begin{equation}\label{numeasure}
\nu_d\{s\in S_d, \mathrm{dist}(s,S\Sigma_d)\leq r_d\}\leq c'N_dd^nr_d
\end{equation}
for some constant $c'>0$ (independent of $d$), where $N_d$ is the dimension of $ \R H^0(X,\oplus_{i=1}^mL^d)$. By Riemann-Roch Theorem, we have that the dimension $N_d$  of $ \R H^0(X,\oplus_{i=1}^mL^d)$ is  $O(d^n)$, so that the right-hand side of \eqref{numeasure} is $O(r_dd^{2n})$, which gives us \eqref{numeasure0} and, then, the result.
\end{proof}

\subsection{Quantitative stability of real sections} 
In this section, we study how much we can perturb a real section $s\in  \R H^0(X,\oplus_{i=1}^mL_i^d)\setminus \Sigma_{d}$ without changing the topology of its real locus. This is the content of Corollary \ref{isotopy} which will use the estimates on the distance to the real discriminant   $\Sigma_d$ proved in Section \ref{sec distance}. In the case of Kostlan polynomials similar results can be found in \cite[Proposition 3]{diatta} and \cite[Theorem 7]{breiding}.
\begin{lemma}\label{thomisotopy} Let $s\in  \R H^0(X,\oplus_{i=1}^mL_i^d)\setminus \Sigma_{d}$ be a  real section (see Definition \ref{realdiscriminant} for the definition of the real discriminant $\Sigma_d$). Then, for any real global section $s'\in \R H^0(X,\oplus_{i=1}^mL_i^d)$ such that 
$$\norm{s-s'}_{\mathcal{C}^1(\R X)}< \mathrm{dist}_{\mathcal{C}^1(\R X)}(s,\Sigma_d),$$
 we have that the pairs $(\R X, \R Z_s)$ and $(\R X, \R Z_{s'})$ are isotopic. (Here, $\norm{\cdot}_{\mathcal{C}^1(\R X)}$ is given by Definition \ref{defC1}.)
\end{lemma}
\begin{proof}
Consider the path of real sections $s_t=(1-t)s+ts'$ for $t\in[0,1]$. Then, by the hypothesis $\norm{s-s'}_{\mathcal{C}^1(\R X)}< \mathrm{dist}_{\mathcal{C}^{1}(\R X)}(s,\Sigma_d),$ we have that $s_t$ is a path of real sections vanishing transversally along $\R X$ for any $t\in[0,1]$. By Thom's Isotopy Lemma, this implies that the pairs $(\R X,\R Z_{s_{t_0}})$ and $(\R X,\R Z_{s_{t_1}})$ are isotopic for any $t_0,t_1\in[0,1]$. Taking $t_0=0$ and $t_1=1$ we have the result.
\end{proof}
\begin{cor}\label{isotopy}  There exists a positive integer $d_0$ such that for any $d\geq d_0$ and any real section $s\in \R H^0(X,\oplus_{i=1}^mL_i^d)\setminus \Sigma_d$, the following happens. For  any real  section $s'\in\R H^0(X,\oplus_{i=1}^mL_i^d)$ such that $$\norm{s-s'}_{\mathcal{C}^1(\R X)}< \mathrm{dist}_{\Sigma_d}(s),$$ we have that the pairs $(\R X,\R Z_s)$ and $(\R X, \R Z_{s'})$ are isotopic. (Here, $\norm{\cdot}_{\mathcal{C}^1(\R X)}$ is given by Definition \ref{defC1}.)
\end{cor} 
\begin{proof}
The result follows directly from Lemmas \ref{distance to the discriminant} and \ref{thomisotopy}.
\end{proof}
\begin{lemma}\label{rapiddecay} Let $\sigma\in  \R H^0(X,\oplus_{i=1}^mL_i^k)$ be a section given by Definition \ref{vanishing order}, for some fixed $k$ large enough.  Then, we have the following estimates as $d\rightarrow\infty$.
\begin{enumerate}
\item There exists $t_0>0$ such that, for any $t\in (0,t_0)$,  we have $$\mu_d\big\{s\in  \R H^0(X,\oplus_{i=1}^mL_i^d), \norm{s_{\sigma^{\lfloor t d\rfloor}}^{\perp}}_{\mathcal{C}^1(\R X)}<\mathrm{dist}_{\Sigma_d}(s)\big\}\geq 1-O(d^{-\infty}).$$
\item There exists a positive $c>0$ such that$$\mu_d\big\{s\in \R H^0(X,\oplus_{i=1}^mL^d), \norm{s_{\sigma}^{\perp}}_{\mathcal{C}^1(\R X)}<\mathrm{dist}_{\Sigma_d}(s)\big\}\geq 1-O(e^{-c\sqrt{d}\log d}).$$ 
Moreover, if the real Hermitian metrics on $L_1,\dots,L_m$ are  analytic, then  the last measure is even bigger than $1-O(e^{-cd})$.
\end{enumerate}
\end{lemma}
\begin{proof}
First, remark that, by Proposition \ref{volume of tubular neighborhood}, for any $m\in\mathbb{N}$, setting $r_d=C_1d^{-2n-m}$, we have 
\begin{equation}\label{rapid1}
\mu_d\big\{s\in\R H^0(X,\oplus_{i=1}^mL_i^d), \mathrm{dist}_{\Sigma_d}(s)> r_d\norm{s}_{\mathcal{L}^2}\big\}\geq 1-O(d^{-m}).
\end{equation}
Also, by  point \textit{(1)} of Proposition \ref{estimatepartial}, for any $t<t_0$, any integer $r$,   any sequence $w_d$ of the form $C_2d^{-r}$, any  $d$ large enough and any real section $s\in\R H^0(X,\oplus_{i=1}^mL_i^d)$, we have
\begin{equation}\label{rapid2}
\norm{s_{\sigma^{\lfloor t d\rfloor}}^{\perp}}_{\mathcal{C}^1(\R X)}< w_d\norm{s}_{\mathcal{L}^2}.
\end{equation}
Putting together \eqref{rapid1} and \eqref{rapid2}, we have that, for any such sequences $r_d$ and $w_d$,
\begin{equation}\label{rapid3}
\mu_d\big\{s\in\R H^0(X,\oplus_{i=1}^mL_i^d), \norm{s_{\sigma^{\lfloor t d\rfloor}}^{\perp}}_{\mathcal{C}^1(\R X)}< \frac{w_d}{r_d}\mathrm{dist}_{\Sigma_d}(s)\big\}\geq 1-O(d^{-m}).
\end{equation}
By choosing $w_d=r_d$, we then obtain that for any $m\in\mathbb{N}$
\begin{equation}\label{rapid3}
\mu_d\big\{s\in\R H^0(X,\oplus_{i=1}^mL^d), \norm{s_{\sigma^{\lfloor t d\rfloor}}^{\perp}}_{\mathcal{C}^1(\R X)}< \mathrm{dist}_{\Sigma_d}(s)\big\}\geq 1-O(d^{-m})
\end{equation}
which proves the point \textit{(1)} of the proposition. Point \textit{(2)} of the proposition follows the same lines, using  point \textit{(2)} of Proposition \ref{estimatepartial} and   setting  $r_d=d^{-2n}e^{-c\sqrt{d}}$, where $c$ is given by point \textit{(2)} of Proposition \ref{estimatepartial}.
\end{proof}
\subsection{Proof of the main theorems} We now prove Theorems \ref{theorem approximation} and \ref{theorem rarefaction}.
\begin{proof}[Proof of Theorem \ref{theorem approximation}]
Let us start with the proof of point \textit{(1)} of the theorem. \\
Let $\sigma\in  \R H^0(X,\oplus_{i=1}^mL_i^k)$ be a section given by Definition \ref{vanishing order}, for some fixed $k$ large enough.
We want to prove that there exists $\alpha_0<1$ such that for any $\alpha> \alpha_0$, the Gaussian measure of the set 
\begin{equation}\label{measuretoestimate}
\big\{s\in \R H^0(X,\oplus_{i=1}^mL_i^d), \exists\hspace{0.5mm} s'\in \R H^0(X,\oplus_{i=1}^mL_i^{\lfloor  \alpha d\rfloor}) \hspace{1mm} \textrm{such that} \hspace{1mm} (\R X, \R Z_s)\sim (\R X, \R Z_{s'}) \big\}
\end{equation}
is at least $1-O(d^{-\infty})$, as $d\rightarrow\infty$, where $(\R X, \R Z_s)\sim (\R X, \R Z_{s'})$ means there the two pairs are isotopic. 

Let us consider the positive real $t_0$ given by Proposition \ref{rapiddecay}\textit{(1)} and set $\alpha_0=1-kt_0$. Then, for any  $\alpha> \alpha_0,$ there exists $t< t_0$, such that the inequality $\lfloor \alpha d\rfloor\geq d-k\lfloor t d\rfloor$ holds. Let us fix such $\alpha$ and $t$.  For any $s\in  \R H^0(X,\oplus_{i=1}^mL_i^d)$, let us write the orthogonal decomposition $s=s_{\sigma^{\lfloor t d\rfloor}}^{\perp}+s_{\sigma^{\lfloor t d\rfloor}}^{0}$ given by Notation \ref{orthogonal component}.
By Lemma \ref{isotopy}, if the $\mathcal{C}^1(\R X)$-norm of $s_{\sigma^{\lfloor t d\rfloor}}$ is smaller than $\mathrm{dist}_{\Sigma_d}(s)$, then the pairs $(\R X, \R Z_s)$ and $(\R X, \R Z_{s_{\sigma^{\lfloor t d\rfloor}}^{0}})$ are isotopic. This implies that the Gaussian measure of the set \eqref{measuretoestimate} is bigger than the Gaussian measure of the set
\begin{equation}\label{measuretoestimate2}
\bigg\{s\in \R H^0(X,\oplus_{i=1}^mL_i^d), \norm{s_{\sigma^{\lfloor t d\rfloor}}^{\perp}}_{\mathcal{C}^1(\R X)}<\mathrm{dist}_{\Sigma_d}(s) \bigg\}
\end{equation}
which, in turn, by  Proposition \ref{rapiddecay}\textit{(1)}, is bigger than $1-O(d^{-\infty})$. We have  proved that the pair $(\R X, \R Z_s)$ is isotopic to the pair $(\R X, \R Z_{s_{\sigma^{\lfloor t d\rfloor}}^{0}})$ with probability $1-O(d^{-\infty})$. Now, the section  $s_{\sigma^{\lfloor t d\rfloor}}^{0}$ lies in the space $\R H_{d,\sigma^{\lfloor t d\rfloor}}$ so that  there exists $s'\in  \R H^0(X,\oplus_{i=1}^mL_i^{d-k\lfloor t d\rfloor})$ such that $s_{\sigma^{\lfloor t d\rfloor}}^{0}=\sigma^{\lfloor t d\rfloor}\otimes s'$. Assertion \textit{(1)} of the theorem then follows from the fact that  the real zero locus of $s_{\sigma^{\lfloor t d\rfloor}}^{0}$ coincides with the real zero locus of $s'$. Indeed,  the real zero locus of $s_{\sigma^{\lfloor t d\rfloor}}^{0}$ equals $\R Z_\sigma\cup \R Z_{s'}$ and this is equal to $\R Z_{s'}$, because  $\R Z_\sigma=\emptyset$.

The proof of the assertion \textit{(2)} of the theorem follows the same lines,  using the orthogonal decomposition  $s=s_{\sigma}^{\perp}+s_{\sigma}^{0}$ and  Proposition \ref{rapiddecay}\textit{(2)}.
\end{proof}

\begin{proof}[Proof of Theorem \ref{theorem rarefaction}]
We start with the proof of Assertion \textit{(1)}.
Recall that we want to prove that for small enough $\epsilon>0$, we have 
$$\mu_d\{s\in\R H^0(X,\oplus_{i=1}^mL^d), b_*(\R Z_s) < (1-\epsilon)b_*(Z_s)\}=1-O(d^{-\infty})$$
as $d\rightarrow\infty$.
 Let $\alpha_0$ be given by point \textit{(1)} of Theorem \ref{theorem approximation} and denote $\delta_0=1-\alpha_0$. By Theorem \ref{theorem approximation}\textit{(1)}, for any $0<\delta<\delta_0$, the real zero locus of a global section $s$ of $L^d$ is diffeomorphic to the real zero locus of a global section $s'$ of $L^{\lfloor (1-\delta)d\rfloor}$  with probability $1-O(d^{-\infty})$. Now, by Smith-Thom inequality (see Equation \eqref{smith-thom}), the total Betti number $b_*(\R Z_{s'})$ of the real zero locus a generic section $s'$ of $L^{\lfloor (1-\delta)d\rfloor}$ is smaller or equal than $b_*(Z_{s'})$, which, by Proposition \ref{totalBetti}, has the asymptotic $b_*(Z_{s'})= \mathrm{v}(L_1,\dots,L_m)\big(\lfloor (1-\delta)d\rfloor\big)^n+O(d^{n-1})$. In particular, with probability $1-O(d^{-\infty})$, the  total Betti number $b_*(\R Z_{s})$ of the real zero locus of a section $s$ of $L^d$ is smaller  than $\mathrm{v}(L_1,\dots,L_m)\big(\lfloor (1-\delta/2)d\rfloor\big)^n$, as $d\rightarrow\infty$. Choosing $\epsilon$ so that $(1-\epsilon)>(1-\delta/2)^n$ we have  the result.

\noindent Assertion \textit{(2)} is proved in the same way, using Theorem \ref{theorem approximation}\textit{(2)} instead of Theorem \ref{theorem approximation}\textit{(1)}.
\end{proof}
\bibliographystyle{plain}
\bibliography{biblio}
\end{document}